\documentclass{amsart}

\usepackage{graphicx} 
\usepackage{tikz}
\usetikzlibrary{calc}
\usepackage{graphicx}
\usepackage{booktabs}
\usepackage{amssymb}
\usepackage{amsmath}
\usepackage{amsthm, amsfonts, mathrsfs}
\usepackage{hyperref}
\newcommand{\Sp}{\mathbb{S}}
\usepackage{comment}
\newtheorem{theorem}{Theorem}
\newtheorem{lemma}[theorem]{Lemma}

\newtheorem{corollary}[theorem]{Corollary}

\newtheorem{question}[theorem]{Question}

\numberwithin{equation}{section}

\DeclareMathOperator{\Ind}{Ind}
\DeclareMathOperator{\VR}{VR}
\DeclareMathOperator{\Cov}{Cov}
\DeclareMathOperator{\supp}{supp}
\DeclareMathOperator{\st}{st}
\DeclareMathOperator{\lk}{lk}
\DeclareMathOperator{\Cl}{Cl}
\DeclareMathOperator{\im}{im}

\newcommand{\KG}{\mathrm{KG}}
\newcommand{\Mat}{\mathcal{M}}
\newcommand{\Nerve}{\mathcal{N}}
\newcommand{\Dow}{\mathcal{D}}
\newcommand{\Q}{\mathbb{Q}}
\newcommand{\Z}{\mathbb{Z}}
\newcommand{\F}{\mathbb{F}}
\newcommand{\tH}{\widetilde{H}}
\newcommand{\K}{\mathbb{K}}
\newcommand{\R}{\mathbb{R}}

\colorlet{dhcol}{violet!75!black}   
\colorlet{dvcol}{teal!70!black}     
\colorlet{boxfill}{gray!8}

\title{The rational homology of independence complexes of Kneser graphs on triples}
\date{\today}
\author[Z. Feng]{Ziqin Feng}
\address{Department of Mathematics and Statistics\\Auburn University\\Auburn\\AL~36849\\USA}
\email{zzf0006@auburn.edu}


\begin{document}

\maketitle

\begin{abstract}
The independence complex of the Kneser graph $\KG(n,k)$ is the simplicial complex
of pairwise intersecting families of $n$-subsets of $[2n+k]$; equivalently, it is
a Vietoris--Rips complex of the $n$-subsets of $[2n+k]$ under the
symmetric-difference metric. Feng and Wang asked whether, for $n=3$, its reduced
homology is nontrivial only in dimensions $6$ and $9$, and conjectured the rank in
dimension $9$. We answer the question affirmatively over the rationals and determine all the associated Betti numbers.  For every $k\ge0$, the only nonzero reduced Betti
numbers using rationals of $\Ind(\KG(3,k))$ are
\[
  \beta_6=\binom{k+5}{7}+29\binom{k+6}{7}
  \qquad\text{and}\qquad
  \beta_9=\binom{k+6}{6}.
\]
The proof uses  an inductive decomposition based on a simplicial retraction
and Barmak's star clusters, Dowker duality together with Jonsson's computation of
the homology of complexes of graphs with bounded covering number, and  a relative
Mayer--Vietoris spectral sequence which reduces the problem to the six-element
sets.
\end{abstract}
 
\maketitle
 
\section{Introduction}\label{sec:intro}
 
For integers $n\ge1$ and $k\ge0$, the \emph{Kneser graph} $\KG(n,k)$ has the
$n$-subsets of $[2n+k]=\{1,\dots,2n+k\}$ as vertices, two of them being adjacent
when they are disjoint. Kneser graphs are a classical meeting point of
combinatorics and topology where Lov\'asz proved Kneser's conjecture by  determining their chromatic number
\cite{lovasz1978kneser} using the Borsuk-Ulam theorem, a classic result in topology. 
This paper investigates the \emph{independence complexes} of Kneser graph. A set of vertices of
$\KG(n,k)$ is independent exactly when its members pairwise intersect, so
$\Ind(\KG(n,k))$ is the complex of pairwise intersecting families of $n$-subsets
of $[2n+k]$, the objects of the Erd\H{o}s--Ko--Rado theory of intersecting
families \cite{godsil2016erdos} when viewed as a simplicial complex.
 
The same complex also arises in applied topology. If the $n$-subsets of $[2n+k]$ are
given the symmetric-difference metric, two of them are at distance at most
$2(n-1)$ precisely when they intersect, so $\Ind(\KG(n,k))$ is the Vietoris--Rips
complex of this finite metric space at scale $2(n-1)$ as noted in \cite{FW}.
Vietoris--Rips complexes of metric spaces are a central object of
topological data analysis, and in recent years a lot of attention
has been drawn to study the homotopy types of Vietoris-Rips complexes of different
metric spaces (\cite{adamaszek2022vietoris,adamaszek2020homotopy,adamaszek2017homotopy, adams2018vietoris, feng2024vietoris,  gasparovic2018complete}, etc).

 Barmak in \cite{barmak2013star} proved that Ind(KG($2, k$)) is homotopy equivalent to the wedge sum of ${k+3\choose 3}$-many copies of $\Sp^2$. In~\cite{FW}, Wang and the author studied these complexes through the Vietoris--Rips
description, introducing \emph{concentration maps} between the underlying metric
spaces. For $n=3$ they computed the Betti numbers using coefficient $\F_2$ of
$\Ind(\KG(3,k))$ for $k\le3$, obtaining $\beta_6=0,29,233,1052$ and
$\beta_9=1,7,28,84$, and found no homology in any other dimension. On the basis of
these values they conjectured that $\beta_9=\binom{k+6}{6}$ for every $k$, and
asked whether dimensions $6$ and $9$ are the only ones carrying nontrivial reduced
homology \cite[Question~2]{FW}. Our main result, Theorem~\ref{thm:main}, answers
this question affirmatively. We prove  the conjectured value of $\beta_9$, and
determine $\beta_6$ as well. In fact we show that for every $k\ge0$ the reduced rational homology of
$\Ind(\KG(3,k))$ vanishes outside dimensions $6$ and $9$, and
\[
  \beta_6=\binom{k+5}{7}+29\binom{k+6}{7},
  \qquad
  \beta_9=\binom{k+6}{6}.
\]
For $k=0$ the complex is the boundary of a $10$-dimensional cross-polytope, so
$\Ind(\KG(3,0))\cong S^9$ and $\beta_6=0$; homology in dimension six appears from
$k=1$ onwards. The formula reproduces all the values computed in~\cite{FW}.
 
\subsection*{Strategy of the proof}
We'll give a brief discussion of the proof strategy. Write $X_m=\Ind(\KG(3,m-6))$ for the complex of pairwise intersecting families of
triples of $[m]$. The proof is an induction on $m$, starting from $X_6\cong S^9$,
in which the inductive step is reduced, through a sequence of identifications, to homology computations on six points. The proof is divided into the following steps. 
 
\emph{The inductive decomposition} (Section~\ref{sec:decomposition}). Deleting the
element $m$ relates $X_m$ to $X_{m-1}$. A concentration map of~\cite{FW} makes
$X_{m-1}$ a simplicial retract of $X_m$ (Lemma~\ref{lem:concentration}), and the
triples containing $m$ span a simplex whose star cluster, in the sense of
Barmak~\cite{barmak2013star}, is contractible. A Mayer--Vietoris argument then
yields split short exact sequences
\[
  0\longrightarrow\tH_q(X_{m-1};\Q)\longrightarrow\tH_q(X_m; \Q)\longrightarrow
  \tH_{q-1}(L_{m-1}; \Q)\longrightarrow0
\]
(Lemma~\ref{lem:decomposition}), where $L_{m-1}$ is the complex of pairwise
intersecting families of triples of $[m-1]$ with transversal number at most two. Since this short exact sequence splits, everything reduces to the homology of the complexes $L_\ell$. 
 
\emph{A complex with known homology} (Section~\ref{sec:decomposition}). Dropping the
intersecting condition gives a larger complex $D_\ell\supseteq L_\ell$, consisting
of all families of triples of $[\ell]$ with transversal number at most two. By
Dowker duality $D_\ell$ is homotopy equivalent to the complex $\Cov(\ell,3)$ of
graphs on $[\ell]$ with covering number at most three
(Lemma~\ref{lem:dowker}), whose homology was determined by Jonsson
(Theorem~\ref{thm:jonsson}). It remains to understand the pair $(D_\ell,L_\ell)$.
 
\emph{Localisation at six-element sets} (Section~\ref{sec:local}). Covering
$D_\ell$ by the simplices of triples meeting a fixed pair produces a relative
Mayer--Vietoris spectral sequence whose $E^1$-page is indexed by graphs on
$[\ell]$. The key structural fact is a classification
(Lemma~\ref{lem:classification}): only three kinds of graphs contribute, all
supported on exactly six vertices --- $K_{3,3}$ minus a matching, $K_2\sqcup C_4$
and $3K_2$. Consequently the relative homology of $(D_\ell,L_\ell)$ is the direct
sum of local contributions, one for each six-element subset of $[\ell]$
(Lemma~\ref{lem:sixset}), and on a single six-element set the second page is
computed explicitly (Lemma~\ref{lem:localE2}); there the contribution of the
graphs $K_{3,3}\setminus M$ is governed by the $3\times3$ chessboard complex.
 
\emph{One finite computation} (Section~\ref{sec:sixpoint}). The spectral sequence
determines the local relative homology only up to the Euler characteristic. The
missing input is the vanishing $\tH_6(L_6;\Q)=\tH_7(L_6;\Q)=0$, which we obtain by
an exact computation of boundary ranks over $\F_2$ (Lemma~\ref{lem:sixpoint}); this
is the only computer-assisted step of the proof, and it is documented in
Appendix~\ref{sec:appendix}. It determines $\tH_*(L_\ell;\Q)$ for every $\ell$
(Corollary~\ref{cor:LN}), and summing the resulting contributions over the
induction gives Theorem~\ref{thm:main} (Section~\ref{sec:proof}).
 
The argument works over $\Q$ throughout. By Jonsson's theorem the intermediate complex $D_\ell$ has $2$-torsion in its integral homology, and whether torsion survives in $\Ind(\KG(3,k))$ is one of the questions stated in Section~\ref{sec:questions}, together with the homotopy type of these complexes and the case $n\ge4$.
 

\section{Notations and Preliminaries}

\subsection{Simplicial complexes}

An \emph{abstract simplicial complex} $K$ on a finite vertex set $V$ is a family of subsets of $V$ such that $\tau\subseteq\sigma\in K$ implies $\tau\in K$. We
include the empty face in every complex. The elements of a simplicial complex $K$ are called simplexes. 

The \emph{dimension} of a simplex with $j+1$ vertices is denoted to be $j$, and  a maximal simplex is called a \emph{facet}. For convenience, we define the binomial coefficients
$\binom{a}{b}$ to be zero whenever $b<0$ or $b>a$.

For a simplex $\sigma\in K$, the \emph{closed star} and the \emph{link}  of $\sigma$ are defined to be
\[
  \st_K(\sigma)=\{\eta\in K:\eta\cup\sigma\in K\},\qquad
  \lk_K(\sigma)=\{\eta\in\st_K(\sigma):\eta\cap\sigma=\emptyset\}.
\]

A vertex
$v$ is a \emph{cone vertex} of a complex $K$ if $\sigma\cup\{v\}\in K$ for every
$\sigma\in K$; then $K=\st_K(v)$ is contractible. A \emph{full subcomplex} is the subcomplex induced by a subset of the vertex
set. A complex is \emph{flag} if every clique of its $1$-skeleton spans a simplex;
$\Cl(G)$ denotes the clique complex of a graph $G$.

The \emph{star cluster} of $\sigma$ is the subcomplex
\[
  \operatorname{SC}_K(\sigma)=\bigcup_{v\in\sigma}\operatorname{st}_K(v).
\] Barmak \cite{barmak2013star} shows that in a flag complex the star cluster of any simplex is contractible.

For complexes $K,L$ on
disjoint vertex sets, we denote $K*L=\{\sigma\cup\tau:\sigma\in K,\ \tau\in L\}$.  The boundary of the
$t$-dimensional cross-polytope, whose $2t$ vertices form $t$ antipodal pairs and
whose simplexes contain at most one vertex from each pair, is the join of $t$ copies
of $S^0$. Hence, 
\begin{equation}\label{eq:crosspoly}
  \partial C_t=(S^0)^{*t}\cong S^{t-1}.
\end{equation}

Unless stated otherwise, homology has rational coefficients, and so the $i$th Betti number in default is $\beta_i(K)=\dim_\Q\tH_i(K;\Q)$. We define that  $\tH_{-1}(K)=0$ for every nonempty complex $K$. 

For a pair $(K,L)$ of simplicial complexes with $L\subseteq K$,   the long exact
sequence of the pair \cite{Hatcher} is given below
\begin{equation}\label{eq:les}
\cdots\longrightarrow\tH_q(L)\longrightarrow\tH_q(K)\longrightarrow H_q(K, L)\longrightarrow\tH_{q-1}(L)\longrightarrow \cdots.
\end{equation}

For a pair $(K, L)$ of finite complexes,
we denote $\chi(K,L)=\sum_i(-1)^i\dim H_i(K,L)$ and  $\widetilde\chi(K)=\sum_i(-1)^i\beta_i(K)$. 

We'll use the Nerve Lemma in the following form \cite{Bjorner}.  If
$\mathcal{U}=\{K_a\}_{a\in I}$ is a finite family of subcomplexes with
$K=\bigcup_aK_a$ and every nonempty finite intersection of members of
$\mathcal{U}$ contractible, then $K\simeq\Nerve(\mathcal{U})$.  Here, the nerve
$\Nerve(\mathcal{U})$ is the complex on $I$ whose faces are the sets
$J\subseteq I$ with $\bigcap_{a\in J}K_a\neq\emptyset$. 

\subsection{Independence complexes and Kneser graphs}

The Kneser graph $\KG(n, k)$ is the graph with vertex set being the $n$-subsets of $[2n+k]$ and that two vertices $A$ and $B$ being adjacent if $A\cap B=\emptyset$. For a finite graph $G$, the \emph{independence complex} $\Ind(G)$ is the complex of independent vertex sets of $G$, here a set of vertices is \emph{independent} if it contains no edge. Hence a collection of the vertices in $\KG(n, k)$ forms a simplex in $\Ind(\KG(n, k))$ if and only if it is pairwise intersecting.   Interestingly, the independence complex of Kneser graph closely relates to the Vietoris-Rips complexes of some metric spaces.

For a finite metric space $(M,d)$ and a scale  $r\ge0$, the \emph{Vietoris--Rips complex} $\VR(M;r)$ has vertex set $M$ so that a finite subset $\sigma\subseteq M$ forms a simplex whenever $d(x,y)\le r$ for all $x,y\in\sigma$. 

For $S\subseteq[m]$ with $|S|\ge n$ let $F^S_n$ be the set of $n$-subsets of
$S$, equipped with the symmetric-difference metric
$d(A,B)=|A\bigtriangleup B|$. For $n$-sets $A$ and $B$,
  $d(A,B)=2(n-|A\cap B|)$. Hence  $d(A,B)\le2(n-1)$ if and only if  $A\cap B\neq\emptyset$. Consequently
\begin{equation}\label{eq:VR}
  \Ind\bigl(\KG(n,k)\bigr)=\VR\bigl(F^{[2n+k]}_n;2(n-1)\bigr),
\end{equation}
both complexes being the flag complex of the same graph. 

For the rest of the
paper, we denote the complex \begin{equation}\label{eq:Xm}
  X_m=\VR\bigl(F^{[m]}_3;4\bigr)=\Ind\bigl(\KG(3,m-6)\bigr),\qquad m\ge 6. 
\end{equation}
The simplexes in $X_m$ are exactly the
pairwise intersecting families of triples in $[m]$.

\subsection{Transversals, matchings and covering complexes}

Let $\mathcal{H}$ be a family of subsets of a finite set $V$. A \emph{transversal} of $\mathcal{H}$ is a set $P\subseteq V$ meeting every member of $\mathcal{H}$. And, when $\mathcal{H}$ is nonempty, we define
\[
  \tau(\mathcal{H})=\min\{|P|:P\cap H\neq\emptyset\ \text{for all}\
  H\in\mathcal{H}\}.
\]

When $\mathcal{H}$ is the edge set of a graph $G$, this is the vertex-cover number of $G$.
Note that every superset of a transversal is again a transversal. 

The \emph{support} of a graph $G=(V,E)$ is $\supp(G)=\bigcup_{e\in E}e$, i.e.,  the set
of its nonisolated vertices. We say that $G$ is \emph{supported on} $S$ when
$\supp(G)=S$. 

A \emph{matching} of $G$ is a set of pairwise disjoint edges, \emph{perfect} if it covers every vertex. And we denote  $\nu(G)$ to be the maximum size of a matching. 
From K\"onig's theorem \cite{diestel2017graph} we know that $\tau(G)=\nu(G)$ for bipartite $G$.

We denote $K_{A,B}$ to be the complete bipartite graph with parts $A,B$, $P_s$ and $C_s$ to be the path and cycle on $s$ vertices, and $rK_2$ to be  a disjoint union of $r$ edges. Also $\sqcup$ denotes the disjoint union operator of graphs. 

For a graph $G=(V, E)$, the \emph{matching complex} $\Mat(G)$ has the vertex set $E$ with the matchings of $G$ being simplexes.  We denote
$\Mat_{3,3}=\Mat(K_{3,3})$, the $3\times3$ chessboard complex. 

Finally, we give the definition of covering complexes as below
\[
  \Cov(N,r)=\Bigl\{E\subseteq\tbinom{[N]}{2}:\tau(E)\le r\Bigr\}.  
\]
Note that this is a simplicial complex whose vertices are  the edges in the complete graph on the vertex set $[N]$.

\subsection{Dowker duality}

For a relation $R\subseteq V\times W$ the two \emph{Dowker complexes} are
\[
\begin{aligned}
  \Dow_V(R)&=\{\sigma\subseteq V:\exists\,w\in W\ \forall v\in\sigma,\ vRw\},\\
  \Dow_W(R)&=\{\omega\subseteq W:\exists\,v\in V\ \forall w\in\omega,\ vRw\}.
\end{aligned}
\]
Dowker's theorem \cite{Dowker} gives a homotopy equivalence
\begin{equation}\label{eq:dowker}
  \Dow_V(R)\simeq\Dow_W(R).
\end{equation}

\subsection{The relative \v{C}ech--Mayer--Vietoris spectral sequence}
\label{ss:ss}

Let $(K,L)$ be a finite simplicial pair and $\{K_a\}_{a\in I}$ a finite cover of
$K$ by subcomplexes. Denote that $L_a=L\cap K_a$. Hence $L=\bigcup_aL_a$ and we say
that the pairs $(K_a,L_a)$ cover $(K,L)$. For $\emptyset\neq J\subseteq I$, let
$K_J=\bigcap_{a\in J}K_a$ and $L_J=\bigcap_{a\in J}L_a=L\cap K_J$. Fix a total
order, $<$,  on $I$ and then we get the first-quadrant double complex
\[
  C_{p,q}=\bigoplus_{a_0<\dots<a_p}
  C_q\bigl(K_{\{a_0,\dots,a_p\}},L_{\{a_0,\dots,a_p\}};\mathbb{F}\bigr),
\]
whose vertical differential is the relative boundary and whose horizontal
differential is the alternating sum of the maps induced by deleting one index.

\vspace{0.2in}

\begin{tikzpicture}[x=1mm,y=1mm,
  box/.style={draw,rounded corners=2pt,fill=boxfill,line width=0.3pt,
              minimum width=26mm,minimum height=10mm,font=\small},
  elt/.style={draw,rounded corners=2pt,line width=0.3pt,
              minimum width=36mm,minimum height=8mm,font=\small},
  dh/.style={->,>=stealth,line width=0.5pt,dhcol},
  dv/.style={->,>=stealth,line width=0.5pt,dvcol},
  lab/.style={font=\footnotesize}]
 
 
\foreach \p in {0,1,2} {
  \node[box] (A\p) at (\p*36,0)   {$C_{\p,q}$};
  \node[box] (B\p) at (\p*36,-22) {$C_{\p,q-1}$};
  \node[lab] at (\p*36,-33) {$p=\p$};
  \draw[dv] (A\p.south) -- node[right,lab] {$d_v$} (B\p.north);
}
\foreach \p/\pp in {1/0,2/1} {
  \draw[dh] (A\p.west) -- node[above,lab] {$d_h$} (A\pp.east);
  \draw[dh] (B\p.west) -- node[above,lab] {$d_h$} (B\pp.east);
}
\node[lab,anchor=east] at (-15,0)   {$q$};
\node[lab,anchor=east] at (-15,-22) {$q-1$};

\end{tikzpicture}
\vspace{0.1in}

Here and below we write $[J;z]$, for $J=\{a_0<\dots<a_p\}$ and
$z\in C_q(K_J,L_J;\mathbb{F})$, for the element $z$ regarded as lying in the
summand indexed by $J$; the same chain occurring in two different summands gives
two different elements of $C_{p,q}$, and the bracket keeps them apart. With this
notation
\[
  d_h[J;z]=\sum_{i=0}^p(-1)^i\bigl[J\setminus\{a_i\};\iota_iz\bigr]\quad(p\ge1),
  \qquad d_h=0\quad(p=0),
\]
where $\iota_i\colon(K_J,L_J)\hookrightarrow(K_{J\setminus\{a_i\}},
L_{J\setminus\{a_i\}})$ is the inclusion. With
$d_{\mathrm{tot}}=d_h+(-1)^pd_v$ on $C_{p,q}$, placed in total degree $p+q$, we
have the following standard statement.

\begin{lemma}\label{lem:ss} Let $\F$ be a field. 
$H_n(\operatorname{Tot}C_{\bullet,\bullet}; \F)\cong H_n(K,L;\mathbb{F})$ for every
$n$, and the filtration by horizontal degree gives a convergent first-quadrant
spectral sequence
\begin{equation}\label{eq:ss}
  E^1_{p,q}=\bigoplus_{\substack{J\subseteq I\\|J|=p+1}}
  H_q(K_J,L_J;\mathbb{F})\ \Longrightarrow\ H_{p+q}(K,L;\mathbb{F}),
\end{equation}
with $d^r\colon E^r_{p,q}\to E^r_{p-r,q+r-1}$. The first differential is induced
by $d_h$: since $d_h$ and $d_v$ commute, $d_h$ carries $d_v$-cycles to
$d_v$-cycles and $d_v$-boundaries to $d_v$-boundaries, so it descends to
$d_v$-homology, and $d^1$ is this descended map. Explicitly, for
$J=\{a_0<\dots<a_p\}$ and a class $\zeta\in H_q(K_J,L_J;\mathbb{F})$,
\begin{equation}\label{eq:d1}
  d^1[J;\zeta]=\sum_{i=0}^p(-1)^i
  \bigl[J\setminus\{a_i\};(\iota_i)_*\zeta\bigr],
\end{equation}
For index sets $J'$ and $J$ we write $d^1_{J',J}$ for the \emph{component} of
$d^1$ from the summand indexed by $J$ to that indexed by $J'$, that is, the
composite of $d^1$ with the inclusion of the $J$-summand and the projection onto
the $J'$-summand. By \eqref{eq:d1},
\[
  d^1_{J',J}=\begin{cases}(-1)^i(\iota_i)_*,& J'=J\setminus\{a_i\},\\
  0,&\text{otherwise,}\end{cases}
\]
so each summand has exactly $p+1$ components leaving it, one for each index that
can be deleted. In particular $d^1$ leaves $q$ unchanged.
\end{lemma}

For the homology case see \cite[Ch.~VII]{Brown}. The relative homology case follows by applying it to $K$ and to $L$, covered by the $K_a$ and the $L_a$.  The relative double complex is the quotient of the two absolute double homologies, and then the five lemma \cite{Hatcher} shows the relative augmentation is a quasi-isomorphism as well. The double complex is finite and first-quadrant, so the spectral sequence converges.

We use \eqref{eq:ss} only through the following consequences. First, for every
$n$,
\begin{equation}\label{eq:dimsum}
  \dim H_n(K,L;\mathbb{F})=\sum_{p+q=n}\dim E^\infty_{p,q};
\end{equation}
in particular $H_n(K,L)=0$ whenever $E^2_{p,q}=0$ for all $p+q=n$, since every
later page is a subquotient of $E^2$. Second, passing to the homology of a page
preserves the alternating sum, so
\begin{equation}\label{eq:eulerpage}
  \sum_{p,q}(-1)^{p+q}\dim E^r_{p,q}=\chi(K,L)\qquad\text{for every }r .
\end{equation}
Third, a morphism of such spectral sequences
which is an isomorphism on some page induces an isomorphism of the homology. This is the comparison theorem (see
\cite{McCleary}). All constructions of this subsection are valid over an
arbitrary field, and also over $\Z$.

\section{The inductive decomposition}\label{sec:decomposition}

\subsection{A concentration retraction}

Following Feng--Wang \cite[\S4]{FW}, for $S\subseteq S'\subseteq[m]$ with
$|S|\ge n$ and an $n$-subset $A$ of $S'$,  the \emph{concentration map} $\varphi^{S'}_S\colon F^{S'}_n\to F^S_n$ maps $A$ to an $n$-subset of $S$ which is $A\cap S$ adjoining the $|A\setminus S|$ smallest elements of
$S\setminus A$. In particular
\begin{equation}\label{eq:phifix}
  \varphi^{S'}_S(A)=A\qquad(A\in F^S_n).
\end{equation}
Feng and Wang show that concentration maps are nonexpanding for the
symmetric-difference metric and hence induce simplicial maps of Vietoris--Rips
complexes at every scale. The following special case is used in our discussion. The proof can be found in \cite{FW}. 

\begin{lemma}\label{lem:concentration}
Let $m\ge7$. Then $\varphi=\varphi^{[m]}_{[m-1]}$ is a simplicial map
$X_m\to X_{m-1}$ which fixes every vertex of $X_{m-1}$. Consequently $X_{m-1}$ is a
simplicial retract of $X_m$, and the inclusion $X_{m-1}\hookrightarrow X_m$ induces an injective homomorphism on reduced homology in any degree and over any coefficient ring.
\end{lemma}

\subsection{The decomposition sequence}

Fix $m\ge 7$. Note that  $X_{m-1}$ is the full subcomplex of
$X_m$ spanned by the triples contained in $[m-1]$. For each
$e=\{a,b\}\in\binom{[m-1]}{2}$, denote 
\[
  v_e=\{m,a,b\},\qquad
  U_m=\bigcup_{e\in\binom{[m-1]}{2}}\st_{X_m}(v_e),\qquad
  L_{m-1}=X_{m-1}\cap U_m .
\]

\begin{lemma}[Decomposition sequence]\label{lem:decomposition}
Let $m\ge7$. For every $q$ there is a short exact sequence of
$\Q$-vector spaces
\begin{equation}\label{eq:deletion}
  0\longrightarrow\tH_q(X_{m-1};\Q)\longrightarrow\tH_q(X_m;\Q)
  \longrightarrow\tH_{q-1}(L_{m-1};\Q)\longrightarrow0,
\end{equation}
and the same statement holds over any field. 

Furthermore, such a sequence splits, although
not naturally.
\end{lemma}

\begin{proof}
Note that every vertex of the complex $X_m$ but not in $X_{m-1}$ is a triple $v_e\subset [m]$ containing $m$ and any simplex in 
$X_m$ is either in $Y$ or 
in $\st_{X_m}(v_e)$ for some $v_e$ containing $m$. Hence, 
\[
  X_m=X_{m-1}\cup U_m.
\]

 For pairs $e_1,\dots,e_s$ in $[m-1]$, the triples $\{v_{e_j}: j=1, 2, \ldots, s\}$ are pairwise intersecting and form a simplex in $X_m$. And $U_m$ is the star cluster of the simplex $\{v_e: e\in \binom{[m-1]}{2} \}$ and by Barmak's result in \cite{barmak2013star}, it is contractible. 

Next, we use the relevant part of the reduced Mayer--Vietoris long exact sequence of $X_m=Y\cup U_m$ below, 

\begin{multline*}
\tH_q(L_N)\longrightarrow\tH_q(Y)\oplus\tH_q(U_m)\longrightarrow\tH_q(X_m)\\
  \longrightarrow\tH_{q-1}(L_N)\longrightarrow\tH_{q-1}(Y)\oplus\tH_{q-1}(U_m).
\end{multline*}

Note that $\tH_q(U_m)= 0$ for any $q$ since $U_m$ is contractible. Also  By the map induced by the inclusion
$X_{m-1}\hookrightarrow X_m$ is injective by
Lemma~\ref{lem:concentration}. Exactness therefore forces the  map
from $\tH_q(L_N)$ to $\tH_q(Y)$ to trivial for all $q$. Then the short exact sequence 
\eqref{eq:deletion} follows, and it splits because every short exact sequence of vector spaces splits.
\end{proof}

So in order to find $\tH_q(X_m; \Q)$ for any $q$,  understanding the homology of $L_{m-1}$ is essential. First, we build a result about the description of the complex $L_{m-1}$. Note that the simplexes in $X_m$ are exactly the pairwise intersecting families of triples in $[m]$. Hence if $\sigma$ is a simplex in $L_{m-1}$, then $\sigma\in \st_{X_m}(v_e)$ for some pair $e$ in $[m-1]$ which is equivalent to $\tau(\sigma)\leq 2$. 

\begin{lemma}\label{lem:Lm-1}
The complex $L_{m-1}$ consists  of the families $\sigma$ of triples of
$[m-1]$ satisfying
\begin{equation}\label{eq:LNdesc}
  v\cap v'\neq\emptyset\text{ for }v,v'\in\sigma, \text{ and }
  \tau(\sigma)\le2 .
\end{equation}
\end{lemma}


\subsection{A related complex with known homology}

For $\ell\geq 3$, we define
\begin{equation}\label{eq:DN}
  D_\ell =\Bigl\{\sigma\subseteq\tbinom{[\ell]}{3}:\tau(\sigma)\le2\Bigr\}, 
\end{equation}
which is a simplicial complex.
 And by Lemma~\ref{lem:Lm-1}, $L_{m-1}=X_{m-1}\cap D_{m-1}\subseteq D_{m-1}$.
Note that the inclusion is strict. Note that the disjoint triples $\{1,2,3\}$ and $\{4,5,6\}$ have
the transversal $\{1,4\}$, so they form a simplex in  $D_{m-1}$ but not of $L_{m-1}$ for $m\geq 7$.

\begin{lemma}\label{lem:dowker}
For every $\ell\ge6$, the complex $D_\ell$ is homotopy equivalent to $\Cov(\ell,3)$.
\end{lemma}

\begin{proof}
Let $R$ be the relation between triples $v\in\binom{[\ell]}{3}$ and pairs
$e\in\binom{[N]}{2}$ given by $v\cap e\neq\emptyset$. It is straightforward to verify that $\mathcal{D}_{\binom{[\ell]}{3}}(R) =D_\ell$ and $\mathcal{D}_{\binom{[\ell]}{2}}(R) =\Cov(N,3)$. Hence, by Dowker's theorem \cite{Dowker}, $D_\ell\simeq\Cov(\ell,3)$
\end{proof}

In \cite{JonssonSIAM,JonssonBook}, Jonsson gave a complete description of the homology groups of $\Cov(\ell, 3)$ which is stated below (\cite[Theorem~26.18]{JonssonBook}).

\begin{theorem}[Jonsson]\label{thm:jonsson}
For every $\ell\ge6$,
\[
  \tH_i\bigl(\Cov(\ell,3);\Z\bigr)\cong
  \begin{cases}
    \Z^{\binom{\ell-1}{6}}, & i=5,\\[2pt]
    (\Z/2)^{\binom{\ell}{6}}, & i=6,\\[2pt]
    \Z^{9\binom{\ell}{6}}, & i=7,\\[2pt]
    \Z^{\binom{\ell}{5}}, & i=8,\\[2pt]
    0, & i\notin\{5,6,7,8\}.
  \end{cases}
\]
\end{theorem}

Then tensoring the integral homology with the rational
$\Q$ gives the following corollary by Lemma~\ref{lem:dowker}. 

\begin{corollary}\label{cor:DN}
For every $\ell\ge 6$,
\begin{equation}\label{eq:DNhomology}
  \tH_i(D_\ell;\Q)\cong
  \begin{cases}
    \Q^{\binom{\ell-1}{6}}, & i=5,\\[2pt]
    \Q^{9\binom{\ell}{6}}, & i=7,\\[2pt]
    \Q^{\binom{\ell}{5}}, & i=8,\\[2pt]
    0, & i\notin\{5,7,8\}.
  \end{cases}
\end{equation}
\end{corollary}

\section{The local spectral sequences}\label{sec:local}

We'll use the long exact sequence induced by the pair $(D_{\ell}, L_\ell)$ to determine the homology of $L_\ell$. To do that, the missing information is the relative homology of the pair $(D_{\ell}, L_\ell)$. Next we will determine the relative homology $H_q(D_\ell, L_\ell)$ for any $q$ using the local spectral sequences at sit-point sets.  

Throughout this section we assume that $\ell \ge 6$ is fixed. For $e\in\binom{[\ell]}{2}$,  let
$\Delta_e$ be the full simplex on the triples meeting $e$ and 
also denote $\Lambda_e=L_\ell\cap\Delta_e$. Then, 
\begin{equation}\label{eq:cover}
  D_\ell=\bigcup_{e\in\binom{[\ell]}{2}}\Delta_e,\qquad L_\ell=\bigcup_{e\in\binom{[\ell]}{2}}\Lambda_e
\end{equation}
 Thus the pairs $(\Delta_e,\Lambda_e)$ with $e\in\binom{[\ell]}{2}$ form a cover of  the pair
$(D_\ell,L_\ell)$. For a nonempty $E\subseteq\binom{[\ell]}{2}$, let
\[
  \Delta_E=\bigcap_{e\in E}\Delta_e,\qquad \Lambda_E=\bigcap_{e\in E}\Lambda_e. 
\]
Hence, the spectral sequence \eqref{eq:ss} becomes
\begin{equation}\label{eq:ssDL}
  E^1_{p,q}=\bigoplus_{\substack{E\subseteq\binom{[\ell]}{2}\\|E|=p+1}}
  H_q(\Delta_E,\Lambda_E;\Q)\ \Longrightarrow\ H_{p+q}(D_\ell,L_\ell;\Q).
\end{equation}

A nonempty index set $E\subset\binom{[\ell]}{2}$ can be regarded as the edge set of a graph $G$ on $[\ell]$ and
write
\[
  C_3(G)=\Bigl\{v\in\tbinom{[\ell]}{3}:v\cap e\neq\emptyset\ \text{for every }
  e\in E\Bigr\}
\]
for the set of vertex covers of $G$ of cardinality exactly three, called
\emph{three-covers}. This doesn't mean that the minimum cover size of the graph $G$ is three. Then the following result is a direct implication of the definition of $\Delta_E$ and $\Lambda_E$. 

\begin{lemma}\label{lem:DeltaLambda}
Let $E\neq\emptyset$ with associated graph $G$. Then $\Delta_E$ is the full
simplex on $C_3(G)$, and $\Lambda_E$ is the clique complex on $C_3(G)$ in which
two three-covers are adjacent when they intersect. If $C_3(G)=\emptyset$, then
$C_*(\Delta_E,\Lambda_E)=0$.
\end{lemma}


Next we will classify the homological properties of the pair $(\Delta_E, \Lambda_E)$ using the graph $G$ determined by the index set $E$. 

\begin{lemma}[Classification]\label{lem:classification}
Let $E\neq\emptyset$ with associated graph $G$ with vertex set $[\ell]$. Suppose that
$H_q(\Delta_E,\Lambda_E;\Q)\neq0$ for at least one $q$. 

Then $G$ is supported on
a unique six-element set $S\subset [\ell]$ such that $\tau(G)=3$ and every three-cover of $G$
is independent. Furthermore  $\Lambda_E$ is the boundary complex of a
$t$-dimensional cross-polytope for some $t\ge 1$. Hence, 
\begin{equation}\label{eq:localrel}
  H_r(\Delta_E,\Lambda_E;\Q)\cong
  \begin{cases}\Q,& r=t,\\ 0,& r\neq t. \end{cases}
\end{equation}

 The induced graph $G|_S$, the value of $t$ and the number of edges are
exactly as in Table~\ref{tab:classification}; see also
Figure~\ref{fig:classification}. Conversely, each graph listed there satisfies
\eqref{eq:localrel} with the stated $t$.
\end{lemma}

\begin{table}[ht]
\caption{The graphs with nonvanishing relative homology.}
\label{tab:classification}
\begin{tabular}{lcc}
\toprule
$G|_S$ & $t$ & $|E|$\\
\midrule
$K_{3,3}\setminus M$, $M$ a matching, $0\le|M|\le3$ & $1$ & $9-|M|$\\
$K_2\sqcup C_4$ & $2$ & $5$\\
$3K_2$ & $4$ & $3$\\
\bottomrule
\end{tabular}
\end{table}

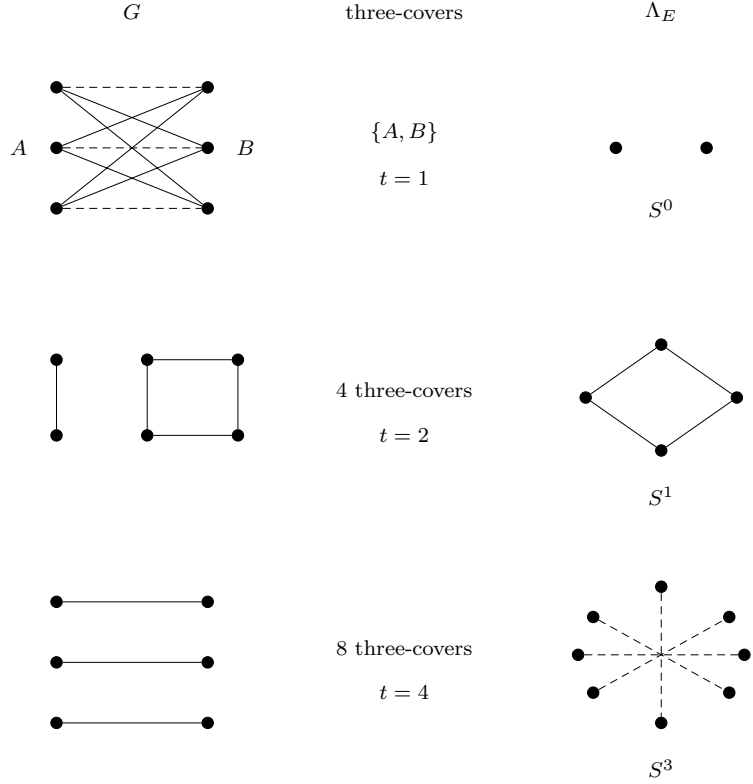
\begin{figure}[ht]
\centering
\begin{tikzpicture}[x=1mm,y=1mm,
  v/.style={circle,draw,fill=black,inner sep=0pt,minimum size=1.5mm},
  e/.style={draw,line width=0.3pt},
  d/.style={draw,line width=0.3pt,densely dashed}]

\node at (10,26) {\footnotesize $G$};
\node at (46,26) {\footnotesize three-covers};
\node at (80,26) {\footnotesize $\Lambda_E$};

\foreach \y in {0,8,16} {\node[v] (a\y) at (0,\y) {}; \node[v] (b\y) at (20,\y) {};}
\draw[e] (a0)--(b8) (a0)--(b16) (a8)--(b0) (a8)--(b16) (a16)--(b0) (a16)--(b8);
\draw[d] (a0)--(b0) (a8)--(b8) (a16)--(b16);
\node at (-5,8) {\footnotesize $A$}; \node at (25,8) {\footnotesize $B$};
\node at (46,10) {\footnotesize $\{A,B\}$}; \node at (46,4) {\footnotesize $t=1$};
\node[v] at (74,8) {}; \node[v] at (86,8) {};
\node at (80,0) {\footnotesize $S^{0}$};

\begin{scope}[yshift=-34mm]
\node[v] (p0) at (0,14) {}; \node[v] (p1) at (0,4) {};
\draw[e] (p0)--(p1);
\node[v] (c1) at (12,14) {}; \node[v] (c2) at (24,14) {};
\node[v] (c3) at (24,4) {}; \node[v] (c4) at (12,4) {};
\draw[e] (c1)--(c2)--(c3)--(c4)--(c1);
\node at (46,10) {\footnotesize $4$ three-covers}; \node at (46,4) {\footnotesize $t=2$};
\node[v] (q1) at (80,16) {}; \node[v] (q2) at (90,9) {};
\node[v] (q3) at (80,2) {}; \node[v] (q4) at (70,9) {};
\draw[e] (q1)--(q2)--(q3)--(q4)--(q1);
\node at (80,-4) {\footnotesize $S^{1}$};
\end{scope}

\begin{scope}[yshift=-68mm]
\foreach \y in {0,8,16} {\node[v] (m\y) at (0,\y) {}; \node[v] (n\y) at (20,\y) {};
  \draw[e] (m\y)--(n\y);}
\node at (46,10) {\footnotesize $8$ three-covers}; \node at (46,4) {\footnotesize $t=4$};
\node[v] (r1) at (80,18) {}; \node[v] (r2) at (80,0) {};
\node[v] (r3) at (71,14) {}; \node[v] (r4) at (89,4) {};
\node[v] (r5) at (69,9) {}; \node[v] (r6) at (91,9) {};
\node[v] (r7) at (71,4) {}; \node[v] (r8) at (89,14) {};
\draw[d] (r1)--(r2) (r3)--(r4) (r5)--(r6) (r7)--(r8);
\node at (80,-6) {\footnotesize $S^{3}$};
\end{scope}
\end{tikzpicture}
\caption{The three classes of graphs of Lemma~\ref{prop:classification} and the associated simplicial 
complexes of three-covers. Dashed edges in the first graph are the removed
matching $M$; dashed segments in the last complex join the four complementary
pairs of covers, which are exactly the non-edges, so that $\Lambda_E$ is a join
of four copies of $S^0$. In each case $\Lambda_E$ is a join of $t$ copies of
$S^0$, hence the boundary of a $t$-dimensional cross-polytope.}
\label{fig:classification}
\end{figure}

\begin{proof}
We will divide the proof into five steps and step 5 is devoted to verify the converse direction. 

\emph{Step 1. We'll identify two disjoint three-covers and a six-point support of the graph $G$.}
If every two three-covers intersect, 
by Lemma~\ref{lem:DeltaLambda} the complex $\Lambda_E$ is then the full simplex
on $C_3(G)$, i.e.,  $\Lambda_E=\Delta_E$, which means all relative homology
vanishes. This contradiction shows that there are disjoint three-covers $A$ and $B$. Then the union
$S=A\sqcup B$ has six elements and every $e\in E$ meets both $A$ and $B$. Thus $E\subseteq E(K_{A,B})$, so $G$
is bipartite with all its nonisolated vertices in $S$.

\emph{Step 2. We'll show that the covering number of $G$ is three and any three-cover of $G$ is independent.}
Since $A$ is a three-cover, $\tau(G)\le3$. Suppose that there is a vertex cover $P$ with
$|P|\le2$.  Pick $f\in E$. Then $P\cap f\neq\emptyset$, so
$|P\cup f|\le3$. Expand $P\cup f$ to a three-element set $C$ which still
a vertex cover with $f\subset C$. Then any other three-cover $C'$ meets
$C$. Since
$\Lambda_E$ is a clique complex, it is a nonempty cone with apex $C$, hence
contractible. Both $\Delta_E$ and $\Lambda_E$ being contractible and nonempty,
the long exact sequence of the pair gives $H_q(\Delta_E,\Lambda_E)=0$ for any $q$ which a
contradiction. Therefore $\tau(G)=3$. Similar argument shows that every three-cover is independent.

\emph{Step 3. We'll show that the simplicial complexes on the three-covers are determined by the complementary pairs and so is the relative homology degree.}
By K\"onig's theorem $\nu(G)=\tau(G)=3$, so $G$ contains a matching
$\{e_1,e_2,e_3\}$. Hence $S=e_1\cup e_2\cup e_3$ and therefore, $\supp(G)=S$, which shows that the support of $G$ is unique. Let $C$ be a three-cover. As $C$ is independent by Step~2, every edge has an
endpoint in $S\setminus C$.  So $S\setminus C$ is a three-cover and hence  is independent as
well. The three-covers therefore occur in complementary pairs
$\{C,S\setminus C\}$. Two three-subsets of the six-set $S$ are disjoint exactly
when they are complementary. Therefore the complex  $\Lambda_E$ is the flag complex of a complete
graph on $2t$ three-covers with complementary pair being nonadjacent, here $t$ is the number
of complementary pairs of three-covers. By \eqref{eq:crosspoly},
\[
  \Lambda_E=(S^0)^{*t}=\partial C_t\cong S^{t-1}. 
\]
 Since $\Delta_E$ is contractible, the long exact sequence of the pair
gives $H_r(\Delta_E,\Lambda_E;\Q)\cong\tH_{r-1}(\Lambda_E;\Q)$, which is
\eqref{eq:localrel}. In particular only when the degree $r=t$, the relative homology is nontrivial. 

\emph{Step 4. We'll give the classification of the induced graph $G|_S$.} First we assume that $G|_S$ is connected. Then it is straightforward to show that its bipartition $A\mid B$ is
unique. We claim the degree of each vertex is $\geq 2$. Suppose not. Pick a vertex $a\in A$ had degree one, with unique neighbour $b\in B$. Since $G|_S$ is 
connected, $b$ has another  neighbour $a'\in A$ of $b$. Then
$C=(A\setminus\{a\})\cup\{b\}$ is a vertex cover. But $C$ is not independent which contradicts the result in Step~2. Hence
every vertex in $G|_S$ has degree at least two. Let $M=K_{A, B}\setminus G|_S$. As $|A|=|B|=3$, $M$ forms a
matching in $K_{A,B}$ and  $0\le|M|\le3$.

Now suppose $G|_S$ is disconnected. Recall that the set $S$ is the union of the perfect matching $e_1, e_2, e_3$ of Step~3. So any connected component has a perfect matching being a subset of $\{e_1, e_2, e_3\}$, and hence every component has even order. Since $G|_S$ is disconnected, the only possible distribution of vertex into components are $4+2$ (two connected components with one having 4 vertices and the other having 2) and $2+2+2$ (three connected components with each having 2 vertices). In the
first case the two-vertex component is $K_2$, and the four-vertex component is a
connected bipartite graph with a perfect matching, hence  is the path $P_4$ or the cycle $C_4$. It
cannot be $P_4$. Otherwise there is a three-cover being non-independent contradicting the result in Step~2.
Thus $G|_S=K_2\sqcup C_4$. In the second case every component is $K_2$, so
$G|_S=3K_2$.

\emph{Step 5. We'll show that each graph in Table~\ref{tab:classification} satisfies
\eqref{eq:localrel} with the stated $t$.}
Let $G|_S=K_{A,B}\setminus M$ with $M$ a matching. Then it is straightforward to verify that any three subset $C$ of $S$ different from $A$ and $B$ can't be a three-cover.  Hence the only three-covers, $A$ and $B$, form a single complementary pair, and $t=1$. For $G=K_2\sqcup C_4$, a three-cover consists of one endpoint
of the $K_2$ together with one of the two opposite pairs of $C_4$. Hence there are four
three-covers in two complementary pairs, so $t=2$. For $G=3K_2$, a
three-cover consists of one endpoint of each edge. Hence there are $2^3=8$ three-covers in
four complementary pairs, and so $t=4$. In each case the three-covers
are independent and \eqref{eq:localrel} holds by Step~3.
\end{proof}

For any subset $S\subseteq[\ell]$ we write $D_S$ and $L_S$ for the complexes
defined by \eqref{eq:DN} and \eqref{eq:LNdesc} with ground set $S$.  Naturally,  $(D_S,L_S)\hookrightarrow(D_N,L_N)$ for any such $S$. Next we'll discuss the local spectral sequences on $(D_S,L_S)$ for $S$ having size $6$. 

\begin{lemma}[The local $E^2$-page]\label{lem:localE2}
Let $S$ be a subset of $[\ell]$ with size $6$. On the $E^2$-page of \eqref{eq:ssDL} for the pair
$(D_S,L_S)$,  the only nontrivial terms are
\begin{equation}\label{eq:E2}
  E^2_{6,1}\cong\Q^{40},\qquad E^2_{4,2}\cong\Q^{45},\qquad
  E^2_{2,4}\cong\Q^{15}.
\end{equation}
In particular $H_i(D_S,L_S;\Q)=0$ for $i\notin\{6,7\}$, and
\begin{equation}\label{eq:localeuler}
  \chi(D_S,L_S)=(45+15)-40=20 .
\end{equation}
\end{lemma}

\begin{proof} Let $E\subset\binom{S}{2}$ be an index set with $|E|=p+1$. On the $E^1$-page of the spectral sequence, the contribution of $E$ is $H_q(\Delta_E,\Lambda_E;\Q)$ on column $p$ and row $q$. Then by Lemma~\ref{lem:classification} this homology is trivial 
unless the associated graph $G$ of $E$ is one of the three types of
Table~\ref{tab:classification}, in which case $H_q(\Delta_E,\Lambda_E;\Q)$ is $\Q$ for $q=t$ and trivial otherwise. 

\emph{Step 1. We'll determine the dimension of the terms, $E_{p, 1}^1$, in the row $q=1$ listed by bipartition.} By Lemma~\ref{lem:classification}, a graph of the first type in Table~\ref{tab:classification} contribute nontrivailly in the corresponding terms. 
Let $G_M=K_{A,B}\setminus M$ be a graph of first type where $A|B$ is an unordered bipartition of $S$ with $|A|=|B|=3$ and $M$ is a matching of $K_{A,B}$. Denote the edge of $G_M$ to be $E(G_M)$ There are $\tfrac12\binom63=10$
bipartitions. If $|M|=r$ then $G_M$ has $9-r$ edges, so its homology $H_1(\Delta_E,\Lambda_E;\Q)$ sits in
column
\begin{equation}\label{eq:column}
  p=|E(G_M)|-1=8-r ,
\end{equation}
and $r$ ranges over $0,1,2,3$. A matching of size $r$ in $K_{A,B}$ is determined
by the following formula, 
\begin{equation}\label{eq:matchcount}
  \binom3r^{\!2}r!\;=\;1,\ 9,\ 18,\ 6\qquad(r=0,1,2,3)
\end{equation}
matchings of size $r$. Hence the row $q=1$ of the $E^1$-page has dimensions
\[
  \dim E^1_{8,1}=10,\qquad \dim E^1_{7,1}=90,\qquad
  \dim E^1_{6,1}=180,\qquad \dim E^1_{5,1}=60,
\]
and is trivial in all other columns.

\emph{Step 2. We'll discuss the components of $d^1$.} For any matching $M$ of $K_{A, B}$,  by Step~5 of
Lemma~\ref{lem:classification} the three-covers of every $G_M=K_{A, B}\setminus M$ are  the complimentary pair $A$
and $B$. Then by Lemma~\ref{lem:DeltaLambda} the pair
$(\Delta_{E(G_M)},\Lambda_{E(G_M)})$ is the $1$-simplex on $\{A,B\}$ relative to
its two endpoints. Its homology $H_1 (\Delta_{E(G_M)},\Lambda_{E(G_M)})\cong\Q$ is generated by
the oriented edge from $A$ to $B$. We write $u_M$ for this class in the
summand with index $E(G_M)$. The bipartition is unordered, so we fix once and
for all, for each bipartition, which triple is called $A$.
 
Consider the component of the differential $d^1$ that deletes an edge $e\in E(G_M)$. It preserves homology dimension
$q$, so its target in $E^1$-page is $H_1(\Delta_{E'},\Lambda_{E'};\Q)$ for the graph
$G'=K_{A,B}\setminus(M\cup\{e\})$ with edge set $E'=E(G_M)\setminus\{e\}$. By
Lemma~\ref{lem:classification} this target is nonzero exactly when $G'$ is the first type in Table~\ref{tab:classification}, i.e., when $G'=K_{A',B'}\setminus M'$ for some bipartition
$A'\mid B'$ and matching $M'$.
\begin{itemize}
\item  If $M\cup\{e\}$ is a matching, then the three-covers of $G'=G_{M\cup\{e\}}$ are again $A$ and $B$ by Step 5 of Lemma~\ref{lem:classification}. Hence $(\Delta_{E'},\Lambda_{E'})=(\Delta_{E(G_M)},\Lambda_{E(G_M)})$, the inclusion inducing the component is the identity. And the component maps $u_M$ to $(-1)^{i(e)}u_{M\cup\{e\}}$, where $i(e)$ is the position of $e$ in $E(G_M)$ with respect to the fixed order on $\binom S2$, as in \eqref{eq:d1}.
\item If $M\cup\{e\}$ is not a matching, then the degree of some vertex in $G'$ is at most $1$. Hence $G'$ is not a graph of type 1 ($t=1$) in Table~\ref{tab:classification}. Therefore $H_1(\Delta_{E'},\Lambda_{E'}; \Q)$ is trivial and so is the component of the differential $d^1$.
\end{itemize}

As a summary, in row $q=1$, the only nontrivial columns are the ones $p=5, 6, 7, 8$ each of which can decompose into the direct sum of components determined by bipartitions. And there are ten such bipartitions. By the discussion above, the differential $d^1$ preserves these decompositions. 

\emph{Step 3. We'll find the $E^2$-page of the spectral sequence in the row $q=1$, $E^2_{p, 1}$.} By the discussion in Step 2,  it suffices
to treat one bipartition, whose corresponding component in the direct sum  has dimensions $1,9,18,6$ in columns
$p =8,7,6,5$ by \eqref{eq:matchcount}. By Step~2 its differential sends $u_M$ to
$\sum\pm u_{M\cup\{e\}}$, the sum over the edges $e$ with $M\cup\{e\}$ a
matching, every coefficient being $\pm1$. In column~$8$, the match $M=\emptyset$ and $d^1u_\emptyset$ has
nine nonzero coefficients, so $d^1$ is injective from $E^1_{8, 1}$ to $E^1_{7, 1}$. This implies that the homology is trivial at column $8$.  In column~$5$, a perfect
matching $P$ and an edge $f\in P$ give $d^1u_{P\setminus\{f\}}=\pm u_P$. Because a
two-edge matching of $K_{3,3}$ extends to exactly one perfect matching, $d^1$
maps from $E_{6, 1}^1$ onto $E_{5, 1}^1$ which implies that the homology is trivial at column~$5$.  In column~$7$, the rank of the differential $d^1:E_{8, 1}^1\mapsto E_{7, 1}^1$ is $1$. And now suppose $d^1\sum_e\lambda_eu_{\{e\}}=0$.  Then
comparing coefficients of $u_{\{e,e'\}}$ gives $\lambda_{e'}=\pm\lambda_e$
whenever $e$ and $e'$ are disjoint. As two edges sharing a vertex are both
disjoint from some third edge, all the coefficients  $\lambda_e$ are determined by any one of
them. So the kernel of $d^1 :E_{7, 1}^1\mapsto E_{6, 1}^1$ is also one-dimensional. The homology therefore vanishes also in $7$. Since taking homology preserves the alternating sum of dimensions,
the dimension in only nontrivial column $6$ is $1-9+18-6=4$. Taking direct sum over the ten bipartitions,
$E^2_{6,1}\cong\Q^{40}$ and all other entries of the row $q=1$ are trivial. Note that the proof uses only which coefficients are nonzero, not their signs, and so holds
over any field.

\emph{Step 4. We'll find the $E^2$-page of the spectral sequence in the rows $q=2$ and $q=4$.} For q=2, the graphs $K_2\sqcup C_4$ on $S$ are
obtained by choosing the two vertices of the $K_2$ component in $\binom62$ ways
and one of the three four-cycles on the remaining four labelled vertices. And this gives
$\binom62\cdot3=45$  such graphs. Each has five edges and hence lies in column $p=4$.
By Table~\ref{tab:classification} no admissible graph with $t=2$ has 
number of edges other than $5$ , so all the other columns except $4$ of this row in the first page are trivial.  As $d^1$ preserves the row number $q$,  $E^2_{4,2}\cong\Q^{45}$.

Similarly for $q=4$, the graphs $3K_2$ on $S$ are the perfect matchings of the complete
graph on six vertices. And there are $6!/(2^3\,3!)=15$ of them, each with three edges and so in
column $p=2$. All the other columns  in the row $q=4$ vanish, and hence
$E^2_{2,4}\cong\Q^{15}$.

\emph{Step 5. Lastly, we assemble the $E^2$-page.} The  total degrees of the nontrivial terms in the $E^2$-page are $6+1=7$, $4+2=6$ and $2+4=6$, so by
\eqref{eq:dimsum} the relative homology vanishes outside degrees $6$ and $7$,
and \eqref{eq:eulerpage} gives
$\chi(D_S,L_S)=(-1)^{7}\cdot40+(-1)^{6}\cdot45+(-1)^{6}\cdot15=20$.
\end{proof}

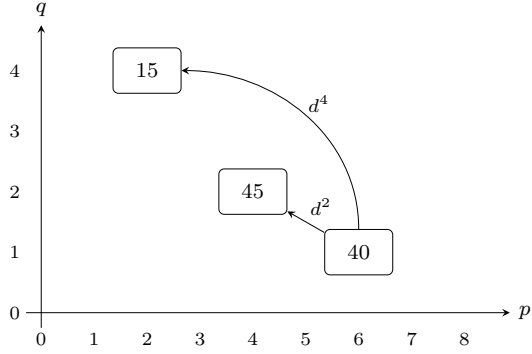
\begin{figure}[ht]
\centering
\begin{tikzpicture}[x=1mm,y=1mm,
  n/.style={draw,rounded corners=2pt,minimum width=9mm,minimum height=6mm,
            font=\footnotesize,line width=0.3pt},
  a/.style={->,>=stealth,line width=0.3pt}]
\draw[->,>=stealth,line width=0.3pt] (-2,0) -- (62,0) node[right,font=\footnotesize] {$p$};
\draw[->,>=stealth,line width=0.3pt] (0,-2) -- (0,38) node[above,font=\footnotesize] {$q$};
\foreach \i/\l in {0/0,1/1,2/2,3/3,4/4,5/5,6/6,7/7,8/8}
  \node[font=\scriptsize] at (\i*7,-3.5) {\l};
\foreach \j/\l in {0/0,1/1,2/2,3/3,4/4}
  \node[font=\scriptsize] at (-3.5,\j*8) {\l};
\node[n] (a) at (42,8) {$40$};
\node[n] (b) at (28,16) {$45$};
\node[n] (c) at (14,32) {$15$};
\draw[a] (a) -- (b);

\node[font=\scriptsize] at (37,14) {$d^2$};
\draw[a] (a.north) to[out=90,in=0]
  node[pos=0.5,above right,font=\scriptsize,inner sep=1pt] {$d^4$} (c.east);
\end{tikzpicture}
\caption{The local $E^2$-page of Lemma~\ref{lem:localE2}. Only three
entries are nonzero: $\Q^{40}$ at $(6,1)$, which lies in total degree seven, and
$\Q^{45}$ at $(4,2)$ and $\Q^{15}$ at $(2,4)$, both in total degree six. Only
two differentials can act between them, and their ranks $a$ and $b$ are not
determined by the spectral sequence: every page has the same Euler
characteristic $20$, so $a+b$ is pinned only by the independent computation of
Lemma~\ref{lem:localrel}.}
\label{fig:E2}
\end{figure}

Only two differentials can act between the three surviving positions in the second page of the spectral sequence as in Figure~\ref{fig:E2}, namely
$d^2\colon E^2_{6,1}\to E^2_{4,2}$ and $d^4\colon E^4_{6,1}\to E^4_{2,4}$. Hence all
$d^3$ vanish, so $E^4=E^3$, and $E^\infty=E^5$. Writing $a=\operatorname{rank}
d^2$ and $b=\operatorname{rank}d^4$ one gets
\[
  \dim H_7(D_S,L_S)=40-a-b,\qquad \dim H_6(D_S,L_S)=60-a-b,
\]
which exhibits the invariant difference $20$ of \eqref{eq:localeuler}.
Lemma~\ref{lem:localrel} below shows that $\dim H_7=9$, hence $a+b=31$;
the individual ranks are not determined here, and no claim of collapse at $E^2$
is made.

\begin{lemma}[Six-set decomposition]\label{lem:sixset}
For every $\ell\ge6$ the inclusions of pairs induce an isomorphism
\begin{equation}\label{eq:sixset}
  \bigoplus_{S\in\binom{[\ell]}{6}}H_i(D_S,L_S;\Q)
  \ \xrightarrow{\ \cong\ }\ H_i(D_\ell,L_\ell;\Q)
\end{equation}
for every $i$.
\end{lemma}

\begin{proof}
Fix a total order on $\binom{[N]}{2}$ and give each $\binom{S}{2}$ the induced
order. For each six-set $S$, the inclusion of pairs induces a morphism from the
spectral sequence \eqref{eq:ssDL} of $(D_S,L_S)$ to that of $(D_N,L_N)$, since
the index set $\binom{S}{2}$ is a subset of $\binom{[N]}{2}$ and the horizontal
differentials are given by the same deletion rule. Taking the direct sum over
all six-sets gives a morphism of spectral sequences.

Every nonzero summand, local or global, is indexed by an edge set $E$ whose graph
$G$ is one of the three types of Lemma~\ref{lem:classification}. In particular
$\supp(G)$ has exactly six elements, and $G$ contains three disjoint edges. Hence
each nonzero summand lies in exactly one six-set, namely $S=\supp(G)$, and every
other index set contributes nothing to any of the spectral sequences. Moreover,
every three-cover of $G$ contains an endpoint of each of three disjoint edges and
so lies in $S$; the three-covers computed in $S$ and in $[N]$ therefore coincide,
and the local and global pairs $(\Delta_E,\Lambda_E)$ are equal. The direct sum
of the local maps thus matches the nonzero summands bijectively, by identity
maps, and hence is an isomorphism on the $E^1$-page. Being induced by inclusions of pairs, it is
a morphism of spectral sequences. Since these are finite, the comparison
theorem gives \eqref{eq:sixset}.
\end{proof}

\section{The six-point computation and the comparison maps}\label{sec:sixpoint}

The following elementary comparison is an immediate consequence of the universal coefficient theorem \cite{Hatcher}. The universal coefficient theorem gives that $$\widetilde H_j(K;\mathbb{F}_2)\cong\bigl(\widetilde H_j(K;\mathbb{Z})\otimes\mathbb{F}_2\bigr)\oplus\operatorname{Tor}\bigl(\widetilde H_{j-1}(K;\mathbb{Z}),\mathbb{F}_2\bigr).$$
Together with $\operatorname{rank}\widetilde H_j(K;\mathbb{Z})=\dim_{\mathbb{Q}}\widetilde H_j(K;\mathbb{Q})$, we get the following result. 

\begin{lemma}\label{lem:rankcomparison}
Let $K$ be a finite simplicial complex. Then
$\dim_\Q\tH_j(K;\Q)\le\dim_{\F_2}\tH_j(K;\F_2)$ for every $j$.
\end{lemma}
We'll discuss the reduced homology of $L_S$ for an arbitrary six set $S$ next.  
\begin{lemma}\label{lem:sixpoint}
For every six-element set $S$,
\begin{equation}\label{eq:sixpointvanishing}
  \tH_6(L_S;\Q)=\tH_7(L_S;\Q)=0 .
\end{equation}

More precisely the reduced Betti numbers of $L_S$ over $\F_2$ are those of
Table~\ref{tab:L6}. Note that they are $0$ except for
$\dim\tH_5(L_S;\F_2)=29$ and $\dim\tH_8(L_S;\F_2)=6$.
\end{lemma}

\begin{table}[ht]
\caption{Exact face numbers, boundary ranks and reduced Betti numbers of $L_6$
over $\F_2$. Here $r_j=\operatorname{rank}_{\F_2}
(\partial_j\colon C_j\to C_{j-1})$ for the augmented complex, and $r_{10}=0$.}
\label{tab:L6}
\begin{tabular}{lrrrrrrrrrr}
\toprule
$j$ & $0$ & $1$ & $2$ & $3$ & $4$ & $5$ & $6$ & $7$ & $8$ & $9$\\
\midrule
$f_j$ & $20$ & $180$ & $960$ & $3360$ & $8064$ & $13410$ & $14640$ & $9900$ &
  $3740$ & $596$\\
$r_j$ & $1$ & $19$ & $161$ & $799$ & $2561$ & $5503$ & $7878$ & $6762$ &
  $3138$ & $596$\\
$\widetilde\beta_j$ & $0$ & $0$ & $0$ & $0$ & $0$ & $29$ & $0$ & $0$ & $6$ &
  $0$\\
\bottomrule
\end{tabular}
\end{table}

\begin{proof}
Relabelling the ground set induces an isomorphism. So we take $S=[6]$. Note that we can split the twenty triples 
into ten complementary pairs. A family of triples is pairwise intersecting
exactly when it selects at most one member from each of these ten pairs. Hence the pairwise intersecting families are
encoded bijectively by the words of $\{\emptyset,+,-\}^{10}$, giving $3^{10}$
candidate faces including the empty one. By Lemma~\ref{lem:Lm-1}, a pairwise intersecting family
$\sigma$ lies in $L_6$ if and only if
\[
  \text{there is }e\in\tbinom{[6]}{2}\text{ with }e\cap v\neq\emptyset
  \text{ for every }v\in\mathcal\sigma. 
\]
 Grouping the retained faces by dimension, indexing the
$(j-1)$-faces, inserting the simplicial boundary of every oriented $j$-face as a
column and row-reducing over $\F_2$ produces the face numbers and ranks of
Table~\ref{tab:L6}. Over $\F_2$ every incidence coefficient is $1$, so a column
is exactly the bit vector of its codimension-one faces and the elimination uses
only exclusive-or on exact bit vectors. No floating-point arithmetic and no
numerical rank decision occur. In degree zero the augmented boundary
$C_0\to C_{-1}\cong\F_2$ is used, so the resulting Betti numbers are obtained from \begin{equation}\label{eq:ranknullity}
  \dim_{\mathbb{F}}\tH_j(K;\mathbb{F})=f_j-r_j^{\mathbb{F}}-r_{j+1}^{\mathbb{F}}.
\end{equation} 
And in particular
\[
\begin{aligned}
  \dim\tH_6(L_6;\F_2)&=14640-7878-6762=0,\\
  \dim\tH_7(L_6;\F_2)&=9900-6762-3138=0 .
\end{aligned}
\]
Lemma~\ref{lem:rankcomparison} now gives \eqref{eq:sixpointvanishing}. The
computation is described in Appendix~\ref{sec:appendix}.
\end{proof}

Next, the local relative homology follows without computing any higher boundary maps.

\begin{lemma}\label{lem:localrel}
For every six-element set $S$,
\begin{equation}\label{eq:localrelhom}
  H_i(D_S,L_S;\Q)\cong
  \begin{cases}
    \Q^{29},& i=6,\\
    \Q^{9},& i=7,\\
    0,& i\notin\{6,7\}.
  \end{cases}
\end{equation}
\end{lemma}

\begin{proof}
From~\ref{lem:localE2} the relative homology of the pair $(D_S, L_S)$ is trivial outside dimensions $6$ and $7$.
For $|S|=6$, $\tH_7(D_S;\Q)\cong\Q^9$ and
$\tH_6(D_S;\Q)=0$ by Corollary~\ref{cor:DN}. The long exact sequence of the pair contains
\[
  \tH_7(L_S)\longrightarrow\tH_7(D_S)\longrightarrow H_7(D_S,L_S)
  \longrightarrow\tH_6(L_S). 
\]
By Lemma~\ref{lem:sixpoint} two outer vector spaces are trivial. By exactness, 
\begin{equation}\label{eq:H7local}
  H_7(D_S,L_S;\Q)\cong \tH_7(D_S) \cong\Q^9 .
\end{equation}
Since the relative homology is concentrated in degrees $6$ and $7$,
by Lemma~\ref{lem:localE2}, 
\[
  \dim H_6(D_S,L_S)-\dim H_7(D_S,L_S)=\chi(D_S, L_S)= 20,
\]
and hence $\dim H_6(D_S,L_S)=29$.
\end{proof}

By Lemma~\ref{lem:sixset} the relative homology of $(D_\ell, L_\ell)$  for $\ell\geq 6$ is the direct sum of one copy of
the local relative group for each of the $\binom{\ell}{6}$ six-element subsets of
$[N]$. By Lemma~\ref{lem:localrel}, we get the following corollary. 

\begin{corollary}[Relative homology]\label{cor:relativeglobal}
For every $\ell\ge6$,
\begin{equation}\label{eq:relglobal}
  H_i(D_\ell, L_\ell;\Q)\cong
  \begin{cases}
    \Q^{29\binom{\ell}{6}},& i=6,\\[2pt]
    \Q^{9\binom{\ell}{6}},& i=7,\\[2pt]
    0,& i\notin\{6,7\}.
  \end{cases}
\end{equation}
\end{corollary}
\begin{lemma}[The degree-seven comparison map]\label{lem:aN}
For every $\ell\ge6$ the  map induced by the quotient map
\[
  \alpha_\ell\colon\tH_7(D_\ell;\Q)\longrightarrow H_7(D_\ell,L_\ell;\Q)
\]
is an isomorphism.
\end{lemma}

\begin{proof}
As proved in Lemma~\ref{lem:localrel} the local map $\alpha_S\colon\tH_7(D_S;\Q)\to H_7(D_S,L_S;\Q)$ is an isomorphism for every six-set $S$. For
each inclusion $S\subseteq[\ell]$, naturality of the long exact sequence of a pair
gives a commutative square in dimension $7$
\[
\begin{array}{ccc}
\tH_7(D_S;\Q) & \xrightarrow{\ \alpha_S\ } & H_7(D_S,L_S;\Q)\\[2pt]
\downarrow & & \downarrow\\[2pt]
\tH_7(D_N;\Q) & \xrightarrow{\ \alpha_N\ } & H_7(D_N,L_N;\Q).
\end{array}
\]
Taking a direct sum of all six-subsets of $[\ell]$, the resulting right-hand vertical map is an isomorphism by Lemma~\ref{lem:sixset}. Since every $\alpha_S$
is surjective, $\alpha_N$ is surjective. Also by Corollary~\ref{cor:DN} and
Corollary~\ref{cor:relativeglobal}, the domain and codomain of $\alpha_N$ have the same
dimension $9\binom{N}{6}$, so $\alpha_N$ is an isomorphism.
\end{proof}

\begin{corollary}[Homology of $L_\ell$]\label{cor:LN}
For every $\ell\ge6$,
\begin{equation}\label{eq:LNhomology}
  \tH_i(L_\ell;\Q)\cong
  \begin{cases}
    \Q^{\binom{\ell-1}{6}+29\binom{\ell}{6}},& i=5,\\[2pt]
    \Q^{\binom{\ell}{5}},& i=8,\\[2pt]
    0,& i\notin\{5,8\}.
  \end{cases}
\end{equation}
\end{corollary}

\begin{proof}
We'll use the long exact sequence of $(D_\ell,L_\ell)$ together with Corollary~\ref{cor:DN} and Corollary~\ref{cor:relativeglobal}.

For $i\leq 4$ $H_{i+1}(D_\ell,L_\ell)$ and $\tH_i(D_\ell)$ are trivial; and hence the exactness of 
$H_{i+1}(D_\ell,L_\ell)\to\tH_i(L_\ell)\to\tH_i(D_\ell)$ yields that
$\tH_i(L_\ell)=0$.

Now take $i=5$. We consider the following part of the long exact sequence
\[
  \tH_6(D_\ell)\to H_6(D_\ell,L_\ell)\to\tH_5(L_\ell)\to\tH_5(D_\ell)\to H_5(D_\ell,L_\ell).
\]
Note that $\tH_6(D_\ell)$ and $H_5(D_\ell,L_\ell)$ are trivial and also
$H_6(D_\ell,L_\ell)\cong\Q^{29\binom{\ell}{6}}$, $\tH_5(D_\ell)\cong\Q^{\binom{\ell-1}{6}}$.
This gives a short exact sequence which splits
\[
  0\longrightarrow\Q^{29\binom{\ell}{6}}\longrightarrow\tH_5(L_\ell;\Q)
  \longrightarrow\Q^{\binom{\ell-1}{6}}\longrightarrow0. 
\]
Hence, $\tH_5(L_\ell;\Q)\cong\Q^{\binom{\ell-1}{6}+29\binom{\ell}{6}}$.

For $i=6, 7$,  the relevant segment of the long exact sequence is
\begin{multline*}
  H_8(D_\ell,L_\ell)\longrightarrow\tH_7(L_\ell)\longrightarrow\tH_7(D_\ell)
  \xrightarrow{\ \alpha_\ell\ }H_7(D_\ell,L_\ell)\\
  \longrightarrow\tH_6(L_\ell)\longrightarrow\tH_6(D_\ell).
\end{multline*}
Note that $H_8(D_\ell,L_\ell)$ and $\tH_6(D_\ell)$ are trivial and  $\alpha_\ell$ is an isomorphism by
Lemma~\ref{lem:aN}, so exactness on both sides of $\alpha_\ell$ gives that 
$\tH_6(L_\ell;\Q)=\tH_7(L_\ell;\Q)=0$.

Finally, for $i\ge8$ one has $H_{i+1}(D_\ell,L_\ell)=H_i(D_\ell,L_\ell)=0$, and so exactness
gives $\tH_i(L_\ell;\Q)\cong\tH_i(D_\ell;\Q)$, which is $\Q^{\binom{\ell}{5}}$ for $i=8$
and zero for $i\ge9$. This finishes the proof. 
\end{proof}

\section{The main theorem}\label{sec:proof}

Now we are ready to identify the homology of $\Ind(\KG(3,k))$. 
\begin{theorem}\label{thm:main}
For every $k\ge 0$,
\[
  \tH_i\bigl(\Ind(\KG(3,k));\Q\bigr)\cong
  \begin{cases}
    \Q^{\binom{k+5}{7}+29\binom{k+6}{7}}, & i=6,\\[2pt]
    \Q^{\binom{k+6}{6}}, & i=9,\\[2pt]
    0, & i\notin\{6,9\}.
  \end{cases}
\]
\end{theorem}

\begin{proof}
We'll start with the base case $X_6$ when $k=0$. The twenty triples of $[6]$ form ten complementary pairs, and two triples are disjoint exactly when they are complementary; hence a family is pairwise intersecting exactly when it contains at most one member of each pair. Therefore,  by \eqref{eq:crosspoly}
\begin{equation}\label{eq:X6}
  X_6=(S^0)^{*10}\simeq S^{9}.
\end{equation}

By Lemma~\ref{lem:decomposition},   for each $j\ge7$, we get the  decomposition
$\tH_q(X_j;\Q)\cong\tH_q(X_{j-1};\Q)\oplus\tH_{q-1}(L_{j-1};\Q)$. Therefore induction on
$j$ yields, for every $m\ge6$,
\begin{equation}\label{eq:induction}
  \tH_q(X_m;\Q)\cong\tH_q(X_6;\Q)\oplus\bigoplus_{\ell=6}^{m-1}\tH_{q-1}(L_\ell;\Q).
\end{equation}
By \eqref{eq:X6} in the base case the homology is nonzero only for $q=9$, and by
Corollary~\ref{cor:LN} a summand involving $L_\ell$ is nontrivial only for
$q-1\in\{5,8\}$, i.e., $q\in\{6,9\}$. Hence,  $\tH_q(X_m;\Q)=0$ for
$q\notin\{6,9\}$.

For $q=6$ the hockey-stick identity gives
\[
\begin{aligned}
  \beta_6(X_m; \Q)&=\sum_{\ell=6}^{m-1}\left[\binom{\ell-1}{6}+29\binom{\ell}{6}\right]
  =\sum_{j=5}^{m-2}\binom{j}{6}+29\sum_{\ell=6}^{m-1}\binom{\ell}{6}\\
  &=\binom{m-1}{7}+29\binom{m}{7}.
\end{aligned}
\]
For $q=9$, we get that
\[
  \beta_9(X_m; \Q)=1+\sum_{\ell=6}^{m-1}\binom{\ell}{5}=1+\left[\binom{m}{6}-1\right]
  =\binom{m}{6}.
\]
Finally since $X_m=\Ind(\KG(3,m-6))$ by \eqref{eq:Xm}, we  substitute $m$ by $k+6$ and get that
\[
  \beta_6(X_m; \Q)=\binom{k+5}{7}+29\binom{k+6}{7},\qquad \beta_9(X_m; \Q)=\binom{k+6}{6},
\]
and this finishes the proof. 
\end{proof}

For any field $\K$ of characteristic zero, the following holds
\[
  \tH_i(X_m;\K)\cong\tH_i(X_m;\Q)\otimes_\Q\K,
\]
since the simplicial chain complex over $\K$ is obtained from the rational one
by tensoring and tensoring vector spaces is exact. Theorem~\ref{thm:main} therefore holds over
every field of characteristic zero, in particular over $\R$ and $\mathbb{C}$.

\section{Open questions}\label{sec:questions}
 
We conclude with some questions that remain open. 
 
\begin{question} [Integral homology]\label{q:integral}
Is $\tH_*(\Ind(\KG(3,k));\Z)$ torsion-free for every $k$?

\end{question}
The question is natural because the computation passes through the complex
$D_N$, which is homotopy equivalent to $\Cov(N,3)$ and whose integral homology
therefore contains $2$-torsion by Theorem~\ref{thm:jonsson}. The argument of this
paper works over $\Q$, where that torsion is ignored. 
 
\begin{question}[Homotopy type]\label{q:homotopy}
Is $\Ind(\KG(3,k))$ homotopy equivalent to a wedge of $\binom{k+5}{7}+29\binom{k+6}{7}$ copies of $S^6$
and $\binom{k+6}{6}$ copies of $S^9$?
\end{question}
 
For $k=0$ this holds, since $\Ind(\KG(3,0))=X_6$ is the boundary of a
$10$-dimensional cross-polytope by \eqref{eq:X6}. In general the degree-nine
classes have natural candidates for spherical representatives: for each
six-element subset $S$ of the ground set, the intersecting families of triples
of $S$ form an embedded copy of $X_6\cong S^9$, and there are exactly
$\binom{k+6}{6}=\beta_9$ of them.

\begin{question}[A basis in degree six]\label{q:basis}
Is there a combinatorial description of a basis of
$\tH_6(\Ind(\KG(3,k));\Q)$?
\end{question}
 
In degree nine the embedded spheres above suggest a basis indexed by the
six-element subsets. In degree six the formula
$\beta_6=\binom{k+5}{7}+29\binom{k+6}{7}$ suggests $29$ classes for each
seven-element subset together with a correction term, but we know of no
combinatorial explanation of the constant $29$.
 
\begin{question}[Larger $n$]\label{q:largern}
For $n\ge4$, what are the dimensions of the nontrivial reduced homology
$\tH_*(\Ind(\KG(n,k));\Q)$ and their ranks?
\end{question}

\section*{Declaration on the use of AI}
In preparing this paper the author used Claude Opus 5.5 and  GPT--6 Astra. They were used to help develop the proofs and identify gaps
and errors, to suggest improvements to the exposition and organization, and to help
typeset the manuscript and figures in \LaTeX. The Python script
\texttt{l6\_check.py}, which reproduces Table~2, was written with its assistance.
All mathematical content, the computations and the references were checked by
the author, who takes full responsibility for the paper.

\appendix
\section{The finite computations}\label{sec:appendix}

Table~\ref{tab:L6} is produced by the Python script \texttt{l6\_check.py},
available as an ancillary file with this paper. It enumerates the faces of $L_6$
as in the proof of Lemma~\ref{lem:sixpoint}, forms the boundary matrices, and
computes their ranks over $\F_2$ by Gaussian elimination on bit vectors; all
arithmetic is exact. The script also performs two consistency checks: the
reduced Euler characteristic is $-23=-29+6$, and the Betti numbers $29$ and $6$
in degrees $5$ and $8$ agree with the values predicted for $\ell=6$ by
Corollary~\ref{cor:LN}, although the proof of Theorem~\ref{thm:main} uses only that the homology is trivial in  degrees $6$ and $7$, in Lemma~\ref{lem:localrel}.

\bibliographystyle{plain}
\bibliography{refeneces_top}

@incollection{Bjorner,
  title={Topological methods},
  author={Bj{\"o}rner, Anders},
  booktitle={Handbook of Combinatorics, {V}ol.~2},
  editor={Graham, Ronald L. and Gr{\"o}tschel, Martin and Lov{\'a}sz, L{\'a}szl{\'o}},
  publisher={Elsevier},
  address={Amsterdam},
  pages={1819--1872},
  year={1995}
}

@article{Dowker,
  title={Homology groups of relations},
  author={Dowker, Clifford H.},
  journal={Annals of Mathematics},
  volume={56},
  pages={84--95},
  year={1952}
}

@misc{FW,
  title={Exploring homological properties of independence complexes of {K}neser graphs},
  author={Feng, Z. and Wang, G.},
  note={Preprint, arXiv:2404.10566},
  year={2024}
}

@book{Hatcher,
  title={Algebraic {T}opology},
  author={Hatcher, Allen},
  publisher={Cambridge University Press},
  address={Cambridge},
  year={2002}
}

@article{JonssonSIAM,
  title={Simplicial complexes of graphs and hypergraphs with a bounded covering number},
  author={Jonsson, Jakob},
  journal={SIAM Journal on Discrete Mathematics},
  volume={19},
  pages={633--650},
  year={2005}
}

@book{JonssonBook,
  title={Simplicial {C}omplexes of {G}raphs},
  author={Jonsson, Jakob},
  series={Lecture Notes in Mathematics},
  volume={1928},
  publisher={Springer},
  address={Berlin},
  year={2008}
}

@book{McCleary,
  title={A {U}ser's {G}uide to {S}pectral {S}equences},
  author={McCleary, John},
  edition={Second},
  series={Cambridge Studies in Advanced Mathematics},
  volume={58},
  publisher={Cambridge University Press},
  address={Cambridge},
  year={2001}
}

@book{Brown,
  title={Cohomology of {G}roups},
  author={Brown, Kenneth S.},
  series={Graduate Texts in Mathematics},
  volume={87},
  publisher={Springer-Verlag},
  address={New York},
  year={1982}
}

@article{barmak2013star,
  title={Star clusters in independence complexes of graphs},
  author={Barmak, Jonathan Ariel},
  journal={Advances in Mathematics},
  volume={241},
  pages={33--57},
  year={2013}
}

@book{godsil2016erdos,
  title={Erd{\H{o}}s--{K}o--{R}ado Theorems: Algebraic Approaches},
  author={Godsil, Christopher and Meagher, Karen},
  series={Cambridge Studies in Advanced Mathematics},
  volume={149},
  publisher={Cambridge University Press},
  address={Cambridge},
  year={2016}
}

@article{lovasz1978kneser,
  title={Kneser's conjecture, chromatic number, and homotopy},
  author={Lov{\'a}sz, L{\'a}szl{\'o}},
  journal={Journal of Combinatorial Theory, Series A},
  volume={25},
  number={3},
  pages={319--324},
  year={1978}
}

@article{adamaszek2022vietoris,
  title={On {V}ietoris--{R}ips complexes of hypercube graphs},
  author={Adamaszek, Micha{\l} and Adams, Henry},
  journal={Journal of Applied and Computational Topology},
  volume={6},
  pages={177--192},
  year={2022}
}

@article{adamaszek2020homotopy,
  title={On homotopy types of {V}ietoris--{R}ips complexes of metric gluings},
  author={Adamaszek, Micha{\l} and Adams, Henry and Gasparovic, Ellen and
          Gommel, Maria and Purvine, Emilie and Sazdanovi{\'c}, Radmila and
          Wang, Bei and Wang, Yusu and Ziegelmeier, Lori},
  journal={Journal of Applied and Computational Topology},
  volume={4},
  pages={424--454},
  year={2020}
}

@article{adamaszek2017homotopy,
  title={On homotopy types of {E}uclidean {R}ips complexes},
  author={Adamaszek, Micha{\l} and Frick, Florian and Vakili, Adrien},
  journal={Discrete \& Computational Geometry},
  volume={58},
  number={3},
  pages={526--542},
  year={2017}
}

@misc{adams2018vietoris,
  title={{V}ietoris--{R}ips complexes of regular polygons},
  author={Adams, Henry and Chowdhury, Samir and Jaffe, Adam Quinn and
          Sibanda, Bonginkosi},
  note={Preprint, arXiv:1807.10971},
  year={2018}
}

@article{feng2024vietoris,
  title={On {V}ietoris--{R}ips complexes of finite metric spaces with scale 2},
  author={Feng, Ziqin and Nukala, N. C. P.},
  journal={Journal of Homotopy and Related Structures},
  volume={19},
  pages={79--98},
  year={2024}
}

@incollection{gasparovic2018complete,
  title={A complete characterization of the one-dimensional intrinsic
         {\v{C}}ech persistence diagrams for metric graphs},
  author={Gasparovic, Ellen and Gommel, Maria and Purvine, Emilie and
          Sazdanovi{\'c}, Radmila and Wang, Bei and Wang, Yusu and
          Ziegelmeier, Lori},
  booktitle={Research in Computational Topology},
  publisher={Springer},
  pages={33--56},
  year={2018}
}

@book{diestel2017graph,
  title={Graph {T}heory},
  author={Diestel, Reinhard},
  edition={Fifth},
  series={Graduate Texts in Mathematics},
  volume={173},
  publisher={Springer},
  address={Berlin},
  year={2017}
}
\end{document}